\documentclass{amsart}

\usepackage{amscd,amssymb, mathrsfs,eucal}
\usepackage[cmtip,arrow,curve,matrix]{xy}

\usepackage{url}
\numberwithin{equation}{section}
\DeclareFontFamily{OMS}{rsfs}{\skewchar\font'60}
\DeclareFontShape{OMS}{rsfs}{m}{n}{<-5>rsfs5 <5-7>rsfs7 <7->rsfs10 }{}
\DeclareSymbolFont{rsfs}{OMS}{rsfs}{m}{n}
\DeclareSymbolFontAlphabet{\scr}{rsfs}

\newtheorem{Theorem}[equation]{Theorem}

\newtheorem{Lemma}[equation]{Lemma}
\newtheorem{Proposition}[equation]{Proposition}
\newtheorem{Corollary}[equation]{Corollary}

\theoremstyle{definition}

\newtheorem{Definition}[equation]{Definition}
\newtheorem{Remark}[equation]{Remark}

\newcommand{\too}{\longrightarrow}

 \newcommand{\ainfty}{A_{\infty}}
\DeclareMathOperator{\Alg}{Alg}
\DeclareMathOperator{\CAlg}{CAlg}

\DeclareMathOperator{\Aut}{Aut}
\DeclareMathOperator{\End}{End}

\newcommand{\B}{\mathcal{B}}

\DeclareMathOperator*{\colim}{colim}

\newcommand{\eqdef}{\overset{\text{def}}{=}}
\newcommand{\einfty}{E_{\infty}}

\newcommand{\EKMM}{\mathcal{M}}

\newcommand{\GLsym}{GL_{1}}
\newcommand{\GL}[1]{\GLsym #1}

\newcommand{\heq}{\simeq}
\DeclareMathOperator{\Ho}{ho}
\renewcommand{\i}{\infty}

\DeclareMathOperator{\id}{id}

\renewcommand{\L}{\mathbb{L}}

\newcommand{\Mgeo}{M_\mathrm{geo}}
\newcommand{\Malg}{M_\mathrm{alg}}

\newcommand{\ptspace}{\ast}

\newcommand{\cf}{\textrm{cf}}

\newcommand{\Rmod}{\Mod{R}}

\newcommand{\slot}{\,-\,}
\newcommand{\spaces}{\TT}

 \newcommand{\sinf}{\Sigma^{\infty}}

\newcommand{\splus}{\sinf_{+}} 
\newcommand{\Smash}{\wedge}

\newcommand{\xra}[1]{\xrightarrow{#1}}

\let\catsymbfont\mathcal

\newcommand{\bL}{{\mathbb{L}}}

\newcommand{\C}{\mathcal{C}}

\newcommand{\D}{\mathcal{D}}
\newcommand{\N}{\mathrm{N}}
\newcommand{\TT}{\mathcal{T}}

\newcommand{\aC}{{\catsymbfont{C}}}
\newcommand{\aM}{\EKMM}

\newcommand{\sL}{{\mathcal{L}}}
\newcommand{\sma}{\wedge}    
\newcommand{\htp}{\simeq}

\newcommand{\sI}{\scr{I}}

\newcommand{\HTT}[1]{[HTT, #1]}
\newcommand{\HA}[1]{[HA, #1]}

\DeclareMathOperator{\Fun}{Fun}

\DeclareMathOperator{\map}{map}

\DeclareMathOperator{\op}{op}
\DeclareMathOperator{\orient}{Orientations}

\DeclareMathOperator{\Cat}{Cat}

\DeclareMathOperator{\Gpd}{Gpd}

\DeclareMathOperator{\Map}{Map}
\DeclareMathOperator{\Pre}{Pre}
\DeclareMathOperator{\Set}{Set}
\DeclareMathOperator{\Stab}{Stab}

\newcommand{\R}{\mathbb{R}}

\newcommand{\Rorient}{\orient_{R}}

\newcommand{\Mod}[1]{{#1}\text{-}\mathrm{mod}}

\newcommand{\Top}{\mathrm{Top}}

\newcommand{\Sing}[1]{\SingOp #1}
\newcommand{\SingOp}{}

\newcommand{\Line}[1]{{#1}\text{-}\mathrm{line}}

\newcommand{\Triv}[1]{{#1}\text{-}\mathrm{triv}}

\newcommand{\Rwe}{\Line{R}}

\newcommand{\Rtriv}{\Triv{R}}

\begin{document}

\title[Thom spectra via $\infty$-categories]{An $\infty$-categorical approach to $R$-line bundles, $R$-module
  Thom spectra, and Twisted $R$-Homology}

\author[Ando]{Matthew Ando}
\address{Department of Mathematics \\
The University of Illinois at Urbana-Champaign \\
Urbana IL 61801 \\
USA} \email{mando@illinois.edu}

\author[Blumberg]{Andrew J. Blumberg}
\address{Department of Mathematics \\University of Texas \\
Austin TX  78712}
\email{blumberg@math.utexas.edu}

\author[Gepner]{David Gepner}
\address{Purdue University\\ West Lafayette IN 47907}
\email{dgepner@purdue.edu}

\author[Hopkins]{Michael J.~Hopkins}
\address{Department of Mathematics \\
Harvard University \\
Cambridge MA 02138}
\email{mjh@math.harvard.edu}

\author[Rezk]{Charles Rezk}

\address{The University of Illinois at Urbana-Champaign \\
Urbana IL 61801 \\
USA}
\email{rezk@illinois.edu}

\thanks{Ando was supported by NSF grants DMS-0705233 and DMS-1104746.
Blumberg was 
  partially supported by an NSF Postdoctoral Research Fellowship and
  by NSF grant DMS-0906105.  Gepner was supported by EPSRC grant
  EP/C52084X/1.  Hopkins was supported by the NSF.  Rezk was supported
  by NSF grants DMS-0505056 and DMS-1006054.}

\begin{abstract}
We develop a generalization of the theory of Thom spectra using the
language of $\i$-categories.  This treatment exposes the
conceptual underpinnings of the Thom spectrum functor: we use a
new model of parametrized spectra,  and our definition is motivated by
the geometric definition of Thom spectra of May-Sigurdsson.  For an $A_\infty$ ring spectrum $R$, we
associate a Thom spectrum to a map of $\i$-categories from the
$\i$-groupoid of a space $X$ to the $\i$-category of free rank
one $R$-modules, which we show
is a model for $B\GL{R}$; we show that $B\GL{R}$
classifies homotopy sheaves of rank one $R$-modules, which we call
$R$-line bundles.  We use our $R$-module Thom spectrum to define the
twisted $R$-homology and cohomology of an $R$-line bundle over a space classified by a map $X\to B\GL{R}$, and we recover the generalized theory of orientations in this context.  In order to
compare this approach to the classical theory, we characterize the
Thom spectrum functor axiomatically, from the perspective of Morita theory.
\end{abstract}

\maketitle

\section{Introduction}
\label{sec:prol-thom-isom}

In the companion to this paper \cite{units-sma}, we review and extend
the work of \cite{MQRT:ersers} on Thom spectra and orientations, using
the theory of structured ring spectra.  To an $A_{\infty}$ ring
spectrum $R$ we associate a space $B\GL{R}$, and to a map of spaces
$f\colon X\to B\GL{R}$ we associate an $R$-module Thom spectrum $Mf$
such that $R$-module orientations $Mf\to R$ correspond to
null-homotopies of $f$.  

Letting $S$ denote the sphere spectrum, one finds that $B\GL{S}$ is
the classifying space for stable spherical fibrations, and if $f$ factors as 
\[
f\colon X \xra{g} B\GL{S} \xra{} B\GL{R},
\]
then $Mg$ is equivalent to the usual Thom spectrum of the
spherical fibration classified by $g$, $Mf \heq Mg \Smash R$, and
$R$-module orientations $Mf\to R$ correspond to classical orientations
$Mg \to R$.

Rich as it is, the classical theory has a number of shortcomings.  For
example, one expects Thom spectra as above to arise from bundles
of $R$-modules.  However, in the
approaches of \cite{units-sma} as well as \cite{MQRT:ersers,LMS:esht},
such a bundle theory is more a source of inspiration than of concrete
constructions or proofs.    A related problem is that, with the
constructions in \cite{units-sma,MQRT:ersers,LMS:esht}, it is
difficult to identify the functor represented by the homotopy type $B\GL{R}.$
The parametrized homotopy theory  of \cite{MR2271789} is one approach to the
bundles of $R$-modules we have in mind, but the material on Thom spectra
in that work focuses on spherical fibrations, and the discussion of 
twisted generalized cohomology in \S22 of that book requires a 
model for $\GL{R}$ which is a genuine topological monoid, a situation which may
not arise from the ambient geometry.

In this paper, we introduce a new approach to parametrized spaces and
spectra, Thom spectra, and orientations, based on the theory of $\i$-categories.  Our treatment has a number of attractive features.  
We use a simple theory of parametrized
spectra as homotopy local systems of spectra.
We give a model for $B\GL{R}$ which, by construction, evidently
classifies homotopy local systems of free 
rank-one $R$-modules.
Using this model, we are able to give an
intuitive and effective construction of the Thom spectrum.
Our Thom spectrum functor is an $\i$-categorical left adjoint,
and so clearly commutes with homotopy colimits, and comes with an
obstruction theory for orientations. We also discuss an axiomatic
approach to the theory of generalized Thom spectra which allows us
easily to check that our construction specializes to the other existing constructions, such as \cite{LMS:esht}.  

To begin, let's consider spaces over a space $X$.  
Since the singular complex functor from spaces to simplicial sets
induces an equivalence of $\i$-categories (where the latter is
equipped with the Kan model structure), and both are equivalent to the
$\i$-category of $\i$-groupoids, we will typically not distinguish
between a space $X$ and its singular complex $\mathrm{Sing}(X)$.  We
will also use the term ``fundamental $\i$-groupoid of $X$'' for any
$\i$-category equivalent to $\mathrm{Sing} (X)$.   In particular, we
may view spaces as $\i$-groupoids, and hence as $\i$-categories.
Moreover, the $\i$-category $\spaces$ of spaces is the prototypical
example of an $\i$-topos, so that for any space $X$, there is a
canonical equivalence
\[
\Fun(X^{\op},\spaces)\longrightarrow{\spaces}_{/X}
\]
between the $\i$-categories of presheaves of spaces on $X$ and spaces
over $X$~\cite[2.2.1.2]{HTT}.  Thus, the $\i$-category $\Fun
(X^{\op},\spaces)$ is a model for the $\i$-category of spaces over
$X$.

Note that, since $X$ is an $\i$-groupoid, there is a canonical
contractible space of equivalences $X\simeq X^{\op}$, and so of 
equivalences
\[
\Fun(X,\spaces)\longrightarrow\Fun(X^{\op},\spaces).
\]
We prefer to use $\Fun (X^{\op},\spaces)$ to emphasize the analogy
with presheaves.

Now let $R$ be a ring spectrum.  We say that an $R$-module $M$
is free of rank one if there is an equivalence of $R$-modules $M
\to R$, and we write $\Line{R}\subset\Mod{R}$ for the subcategory
consisting of the free rank one $R$-modules and the equivalences
thereof.  By construction, $\Line{R}$ is an $\i$-groupoid, i.e., a
Kan complex.   In a precise sense which we now explain, $\Mod{R}$
classifies bundles of $R$-modules, and $\Line{R}$ classifies bundles
of free rank one $R$-modules whose fibers  are glued together by
$R$-linear equvialences.

Given a space $X$, a functor  (i.e., a map of simplicial sets)
\[
L:X^{\op}\longrightarrow\Mod{R}
\]
is a sort of local coefficient system:
for each point $p\in X$, we have a $R$-module $L_p$.
To a path $\gamma: p \to q$ in $X$, $L$
associates an 
equivalence of $R$-modules
\begin{equation}\label{eq:3}
   L_{\gamma}: L_{q} \heq L_{p}.
\end{equation}
From a homotopy of paths $h: \gamma\to \gamma'$, we get 
a path
\begin{equation}\label{eq:4}
    L_{h}: L_{\gamma'} \to L_{\gamma}
\end{equation}
in the space of $R$-module equivalences $L_{p}\to L_{q},$ and so
forth for higher homotopies.  More precisely, $L$ is a
``homotopy local system'' of $R$-modules.  The fact that the data of a
functor from $X^{\op}$ to $\Mod{R}$ are precisely the higher coherence
conditions for a homotopy local system of $R$-modules is what makes
the theory of 
$\i$-categories so effective in this context.  With this in mind, we
make the following definition. 

\begin{Definition}[\S\ref{sec:glr-symm-mono}]
Let $X$ be a space.  
A \emph{bundle of $R$-modules over $X$} is a functor 
\[
   f: X^{\op} \too \Mod{R}.
\]
A \emph{bundle of $R$-lines over $X$} is a functor 
\[
   f: X^{\op} \too \Line{R}.
\]
We write $\Fun(X^{\op},\Rmod)$ and $\Fun(X^{\op},\Rwe)$ for the $\i$-categories of
bundles of $R$-modules and  $R$-lines over $X$.
\end{Definition} 

Our definition of the Thom spectrum is motivated by the May-Sigurdsson
description of the ``neo-classical'' Thom spectrum as the composite of the pullback of a
universal parametrized spectrum followed by the base change along the
map to a point \cite[23.7.1,23.7.4]{MR2271789}.   

\begin{Definition}[\S\ref{sec:thom-r-modules}]\label{def-7}
Let $X$ be a space.  The Thom $R$-module spectrum $Mf$ of a
bundle of $R$-lines over $X$
\[
   f\colon X^{\op} \too \Rwe
\]
is the colimit of the functor
\[
   X^{\op}  \overset{f}{\too} \Rwe \too \Rmod.
\]
obtained by composing with the inclusion $\Rwe\subset\Rmod$.
\end{Definition}

It is very easy to describe the obstruction theory for orientations in this setting.  
The colimit in Definition
\ref{def-7} means that the space of $R$-module maps 
\[
   Mf \too R
\]
is equivalent to the space of maps of bundles of $R$-modules
\[
   f \to R_{X},
\]
where $R_{X}$ denotes the trivial bundle of $R$-lines over $X$, i.e. the constant functor $X^{\op}\to\Rwe$ with value $R\in\Rwe$.

\begin{Definition}[Definition \ref{inf-def-2}]
The space of \emph{orientations} $Mf \to R$ is the pull-back in the
diagram 
\[
\xymatrix{
\orient (Mf,R) \ar[r]\ar[d]_\simeq & \map_{\Mod{R}} (Mf,R)\ar[d]^\simeq \\
\map_{\Line{R_X}}(f,R_{X}) \ar[r] & \map_{\Mod{R_X}} (f,R_{X}).}
\]
That is, orientations $Mf\to R$ are those $R$-module maps which
correspond to trivializations $f \heq R_{X}$
of the bundle of $R$-lines $f$.
\end{Definition}

Put another way, let $\Rtriv$ be the
$\infty$-category of trivialized $R$-lines: $R$-lines $L$ equipped with
a specific equivalence $L \xra{\heq}R$.  $\Rtriv$ is a contractible Kan complex, and
the natural map
\begin{equation}\label{eq:10}
     \Rtriv \rightarrow \Rwe
\end{equation}
is a Kan fibration.  
We then have the following.

\begin{Theorem}[Theorem \ref{t-th-or-thy-infty-lifting}]
If $f\colon \Sing{X^{\op}}\to \Line{R}$ is a bundle of $R$-lines over $X$, then the
space of orientations $Mf \to R$ is equivalent to the space of lifts in
the diagram 
\begin{equation} \label{eq:53}
\xymatrix{ & {\Rtriv}
 \ar[d] \\
{\Sing{X}} \ar@{-->}[ur] \ar[r]_-{f} & {\Rwe}.}
\end{equation}
\end{Theorem}

\noindent Analogous considerations lead to a version of the Thom isomorphism in
this setting.

Finally, using this notion of $R$-module Thom spectrum, we can define
the twisted $R$-homology and $R$-cohomology of a space $f\colon
X\to\Rwe$ equipped with an $R$-line bundle by the formulas 
\begin{align}
R^f_n(X) =&\,\pi_0\map_R(\Sigma^n R,Mf)\cong\pi_n Mf\\
R_f^n(X) =&\,\pi_{0}\map_R(Mf,\Sigma^n R).
\end{align}

In the presence of an orientation, we have the following untwisting
result.

\begin{Corollary}
If $f\colon X^{\op}\to\Rwe$ admits an orientation, then $Mf\simeq
R\land\Sigma^\infty_+ X$, and the twisted $R$-homology and
$R$-cohomology spectra 
\begin{align}
R^f(X)\simeq &\, R\land\Sigma^\infty_+ X\\
R_f(X)\simeq &\, \Map(\Sigma^\infty_+ X,R)
\end{align}
reduce to the ordinary $R$-homology and $R$-cohomology spectra of $X$ (here $\Map$ denotes the function spectrum).
\end{Corollary}

In Section~\ref{sec:comp-thom-spectra} we relate the theory developed
in this paper to other approaches, such as 
\cite{MQRT:ersers,LMS:esht,units-sma}.   We rely on the fact that
$\Line{R}$ is a model for $B\GL{R}$.  Indeed, the fiber of 
\eqref{eq:10} at $R$ is $\Aut_{R} (R)$, by which we mean the derived
space of $R$-linear self-homotopy equivalences of $R$ (e.g.,
see~\cite[\S 2]{units-sma}).  More precisely, we have the following.

\begin{Corollary}[Corollary~\ref{cor:gl-kan-fib}]
The Kan fibration 
\begin{equation}\label{eq:11}
   \Aut_R(R) \xra{} \Rtriv \xra{} \Rwe
\end{equation}
is a model in simplicial sets for the quasifibration $\GL{R}\to
E\GL{R}\to B\GL{R}$.
\end{Corollary}

\begin{Remark}
In fact, since geometric realization carries Kan fibrations to Serre
fibrations~\cite{MR0238322}, upon geometric realization we obtain a
Serre fibration which models  
\[
   \GL{R} \to E\GL{R} \to B\GL{R}
\]
in topological spaces.  The approach taken in
\cite{MQRT:ersers,units-sma} is only known to produce a
quasifibration. 
\end{Remark}

The equivalence $B\Aut (R) \heq \Rwe$
implies the following description of the Thom spectrum functor of
Definition \ref{def-7}, which 
plays a role in \S\ref{sec:comp-thom-spectra} when we compare our
approaches to Thom spectra.  Recall that if $x$ is an object in an
$\i$-category $\aC$, then $\Aut_{\aC} (x)$, is a 
group-like monoidal $\i$-groupoid, that is, a group-like $\ainfty$
space; conversely if $G$ is a group-like monoidal $\i$-groupoid, then
we can form the $\i$-category $BG$ with a single object $*$ and $G$ as
automorphisms.  Moreover, an action of $G$ on $x$
is just a functor $BG\to \aC$.

\begin{Theorem} \label{t-thom-sp-is-R-mod-G-intro}
Let $G$ be a group-like monoidal $\i$-groupoid.  A map $BG\xra{}\Rwe$
specifies an $R$-linear action of $G$ on $R$, and then the Thom
spectrum is equivalent to the 
(homotopy) quotient $R / G$.
\end{Theorem}

The preceding theorem follows immediately from the construction of the
Thom spectrum, since by definition the quotient in the statement is the
colimit of the map of $\i$-categories $BG \xra{} \Rwe \to \Rmod$.

Turning to the comparisons, the definitions of \cite{units-sma} and
this paper give two constructions of an $R$-module Thom spectrum from
a map $f\colon X \to B\GL{R}$.  Roughly speaking, the ``algebraic''
model studied in \cite{units-sma} takes the pull-back $P$ in the diagram 
\[
\xymatrix{
     P \ar[r] \ar[d] & E\GL{R} \ar[d] \\
      X \ar[r]^-{f} & B\GL{R},
}
\]
and sets 
\begin{equation}\label{eq:5}
   \Malg f  = \splus P \Smash_{\splus \GL{R}} R.
\end{equation}
The ``geometric'' model in this paper sets 
\[
    \Mgeo f = \colim (\Sing{X^{\op}} \xra{f} \Sing B\GL{R} \heq \Rwe \to \Rmod).
\]
It is possible to show by a direct calculation that these two
constructions are equivalent; we do this in section
\ref{sec:algebr-thom-spectr}. 
The bulk of \S\ref{sec:comp-thom-spectra} is concerned with a
more general 
characterization of the Thom
spectrum functor from the point of view of Morita theory.
Here we also show that our Thom spectrum recovers the Thom
spectrum of \cite{LMS:esht} in the special case of a map $f\colon X\to B\GL{S}.$ 

In \eqref{eq:5}, the Thom spectrum $\Malg$ is a derived
smash product with $R$, regarded 
as an $\splus \GL{R}$-$R$ bimodule, specified by the canonical action of
$\splus \GL{R}$ on $R$.  Recalling that the target category of
$R$-modules is stable, we can regard this Thom spectrum as
given by a functor from (right) $\splus \GL{R}$-modules to $R$-modules.
Now, roughly speaking, Morita theory (more precisely, the
Eilenberg-Watts theorem) implies that any continuous functor from
(right) $\splus \GL{R}$-modules to $R$-modules which preserves homotopy
colimits and takes $\GL{R}$ to $R$ can be realized as tensoring with
an appropriate $\splus \GL{R}$-$R$ bimodule.  In particular, this
tells us that the Thom spectrum functor is characterized among such
functors by the additional data of the action of $\GL{R}$ on $R$.

We develop these ideas in the setting of $\i$-categories.  Let
$\spaces$ be the $\i$-category of spaces.  Given a colimit-preserving
functor $F\colon \spaces_{/B\Aut (R)} \to \Rmod$ 
which sends $\ptspace/B\Aut (R)$ to $R$,
we can restrict along the Yoneda embedding \eqref{eq:63}
\[
      B\Aut (R) \xra{} \spaces_{/B\Aut (R)} \xra{F} \Rmod;
\]
since it takes the object of $B\Aut (R)$ to $R$, we may view
this as a functor
\[
   k\colon B\Aut (R) \to B\Aut (R).
\]
Conversely, given an endomorphism $k$ of $B\Aut
(R)$, we get a colimit-preserving functor
\[
    F\colon \spaces_{/B\Aut (R)} \to \Rmod
\]
whose value on $B^{\op} \to B\Aut (R)$ is
\[
   \colim (B^{\op} \xra{} B\Aut (R) \xra{k} B\Aut
   (R)\hookrightarrow  \Rmod).
\]
About this correspondence we prove the following.
\begin{Proposition}[Corollary \ref{t-co-axiom-geom-thom-functor}] \label{t-pr-axiom-thom-intro}
A functor $F$ from the $\i$-category $\spaces_{/B\Aut (R)}$ to the
$\i$-category of $R$-modules is equivalent to the Thom spectrum
functor if and only if it
preserves colimits and its restriction along the Yoneda embedding
\[
B\Aut (R) \to \spaces_{/B\Aut (R)} \xra{F} \Rmod
\]
is
equivalent to the canonical inclusion
\[
  B\Aut (R) \xra{\heq} \Rwe \xra{} \Rmod.
\]
\end{Proposition}

It follows easily (Proposition \ref{t-pr-thom-spectra-equivalent}) that
the Thom spectrum functors $\Mgeo$ and $\Malg$ are equivalent.
It also follows that, as in Proposition
\ref{t-thom-sp-is-R-mod-G-intro}, the Thom spectrum of a group-like $\ainfty$ map $\varphi\colon G \to \GL{S}$
is the (homotopy) quotient
\[
    \colim (BG^{\op} \to \Rmod) \heq R/G.
\]
This observation is the basis for our comparison with the Thom
spectrum of Lewis and May.
In \S\ref{sec:neo-classical-thom} we show that the Lewis-May Thom
spectrum associated to the map $B\varphi\colon BG \to B\GL{S}$
is a model for the (homotopy) quotient $S/G$.

\begin{Proposition}[Corollary \ref{t-co-LM-thom-comp}]
The Lewis-May Thom spectrum associated to a map
\[
    f\colon B \to B\GL{S}
\]
is equivalent to the Thom spectrum associated by
Definition~\ref{def-7} to the map of $\infty$-categories 
\[
     \Sing{B}^{\op} \xra{\Sing{f}}  \Sing{B\GL{S}} \heq \Line{S}.
\]
\end{Proposition}

\section{Parametrized spectra and Thom spectra}
\label{sec:units-via-infty}

In this section, we show that the theory of $\i$-categories provides a
powerful technical and conceptual framework for the study of Thom
spectra and orientations.  We chose to use the theory of
quasicategories as developed by Joyal and Lurie~\cite{MR1935979,HTT},
but for the theory of $R$-module Thom spectra and orientations, all
that is really required is a good $\i$-category of $R$-modules.

\subsection{$\i$-Categories and $\i$-Groupoids}

\label{sec:i-categories-i}

For the purposes of this paper, an $\i$-category will always
mean a quasicategory in the sense of Joyal \cite{MR1935979}.  This is
the same as a weak Kan complex in the sense of Boardman and Vogt
\cite{MR0420609}; the different terminology reflects the fact that 
these objects simultaneously generalize the homotopy theories of categories and of
spaces.  There is nothing essential in our use of quasicategories, and any other sufficiently well-developed theory of $\infty$-categories (more precisely, $(\infty,1)$-categories) would suffice.

Given two $\i$-categories $\C$ and $\D$, the $\i$-category of functors
from $\C$ to $\D$ is simply the simplicial set of maps from $\C$ to
$\D$, considered as simplicial sets.  More generally, for any simplicial set $X$
there is an $\i$-category of functors from $X$ to $\C$, written $\Fun(X,\C)$;
by \HTT{Proposition 1.2.7.2, 1.2.7.3}, the simplicial set $\Fun(X,\C)$
is a $\infty$-category whenever $\C$ is, even for an arbitrary simplicial set
$X$.

This description of $\Fun(\C,\D)$ gives rise to simplicial
categories of $\i$-categories and $\i$-groupoids.  For our purposes
it is important to have {\em $\i$-categories} $\Cat_\i$ and $\Gpd_\i$
of $\i$-categories and $\i$-groupoids, respectively.  We construct
these $\i$-categories by a general 
technique for converting a simplicial category to an $\i$-category:
there is a simplicial nerve functor $\N$ from simplicial
categories to $\i$-categories which is the right adjoint of a Quillen equivalence $
\mathfrak{C}:\Set_\Delta\rightleftarrows\Cat_\Delta:\N$ \HTT{\S 1.1.5.5, 1.1.5.12, 1.1.5.13}.
Note that this process also gives rise to a standard passage from a simplicial
model category to an $\i$-category which retains the homotopical information encoded by the simplicial model structure.  Specifically, given a
simplicial model category $\mathcal{M}$, one restricts to the simplicial category on the
cofibrant-fibrant objects, $\mathcal{M}^\cf$.  Then applying the simplicial
nerve yields an $\i$-category $\N(\mathcal{M}^\cf)$.

In particular, $\Cat_\i$ is the simplicial nerve of the
simplicial category of $\i$-categories, in which the mapping spaces
are made fibrant by restricting to maximal Kan subcomplexes, and
$\Gpd_\i$ is the full $\i$-subcategory of $\Cat_\i$ on the
$\i$-groupoids.
We recall that the Quillen equivalence between the standard model
structure on topological spaces and the Kan model structure on
simplicial sets induces an equivalence on underlying $\i$-categories.
Thus, as all the constructions we perform in this paper are homotopy
invariant, we will typically regard topological spaces as
$\i$-groupoids via their singular complexes. 

Let $\C$ be an $\infty$-category.
Then $\C$ admits a maximal $\infty$-subgroupoid $\C^{\simeq}$, which is by definition the pullback (in simplicial sets) of the diagram
\begin{equation}\label{eqn:ho}
\xymatrix{
\C^{\simeq}\ar[r]\ar[d] & \C\ar[d]\\
\N\Ho(\C)^{\simeq}\ar[r] & \N\Ho(\C)\,,}
\end{equation}
where $\N\Ho(\C)$ denote the nerve of the homotopy category of $\C$, and $\N\Ho(\C)^{\simeq}$ is the maximal subgroupoid.
Thus, if $a$ and $b$ are objects of
$\C$, $\C^{\simeq}(a,b)$ is the subcategory of $\C(a,b)$
consisting of the {\em equivalences}.

\subsection{Parametrized spaces}

\label{sec:bundles-i-groupoids}

Let $X$ be an $\i$-groupoid, which we view as the fundamental
$\i$-groupoid of a topological space.
There are two canonically equivalent $\i$-topoi associated to $X$;
namely, the slice $\i$-category $\Gpd_{\i/X}$ of $\i$-groupoids over
$X$, and the $\i$-category $\Fun (X^{\op},\Gpd_{\i})$ of presheaves of
$\i$-groupoids on $X$.  The equivalence 
\[
     \Fun (X^{\op},\Gpd_{\i}) \heq \Gpd_{\i/X}
\]
sends a functor to its colimit, regarded as a space over $X$, and may
be regarded as a generalization of the equivalence between (free)
$G$-spaces and spaces over $BG$~\cite[2.2.1.2]{HTT}.
In particular, a terminal object $1\in\Fun(X^{\op},\Gpd_\i)$ must be sent to a terminal object $\id_X\in\Gpd_{\i/X}$, which in this special case recovers the formula
\begin{equation}\label{colim1=X}
   \colim_{X^{\op}} 1 \simeq X.
\end{equation}

\begin{Remark} 
As explained in the introduction, the data of a functor $L: X^{\op}\to
\Gpd_{\i}$ encodes the data of a homotopy local system of spaces on
$X$. 
\end{Remark}

\begin{Remark} \label{rem-4}
Since $X$ is an $\i$-groupoid, we have
a canonical contractible space of equivalences $X\heq X^{\op}$, which induces equivalences
\[
  \Fun(X,\Gpd_\i)\simeq \Fun (X^{\op},\Gpd_{\i}) \heq \Gpd_{\i/X}.
\]
Although notationally a bit more complicated, we think it is slightly more natural to regard spaces over $X$ as contravariant functors, instead of covariant functors, from $X$ to $\Gpd_\i$.
One reason for this is that, this way, the Yoneda embedding appears naturally as a functor $X\to\Fun(X^{\op},\Gpd_\i)$, and this will play an important role in our treatment of the Thom spectrum functor (cf. Proposition
\ref{t-pr-yoneda-taut}).
\end{Remark}

\begin{Lemma}\label{t-le-colimits-base-change}
The base-change functor
$
f^* \colon \Gpd_{\i/X}\to\Gpd_{\i/X'}
$
admits a right adjoint.
In particular, $f^*$ commutes with colimits.
\end{Lemma}

\begin{proof}
For the proof, see \cite[6.1.3.14]{HTT}.
\end{proof}
\begin{Remark} 
If $X$ is an $\i$-groupoid, then via the equivalence of
$\i$-categories $\Gpd_{\i/X}\simeq\Fun(X^{\op},\Gpd_{\i})$, 
the Yoneda embedding $X\to \Gpd_{\i/X}$ sends the point $x$ of $X$ to
the ``path fibration'' $X_{/x}\to X$.  (This follows from an analysis
of the ``unstraightening'' functor that provides the right adjoint
in~\cite[2.2.1.2]{HTT}.)
\end{Remark}

\subsection{Parametrized spectra}
\label{sec:infinity-category-r}

An $\i$-category $\C$ is {\em stable} if its has a zero object, finite
limits, and the endofunctor $\Omega\colon \C\to\C$, defined by sending
$X$ to the limit of the diagram $\ast\to X\leftarrow\ast$, is an
equivalence~\cite[1.1.1.9, 1.4.2.27]{HA}.
It follows that the left adjoint $\Sigma$ of $\Omega$ is also an
equivalence, that finite products and finite coproducts agree, and
that square $\Delta^1\times\Delta^1\to\C$ is a pullback if and only if
it is a pushout (so that $\C$ has all finite colimits as well). A
morphism of stable $\i$-categories is an exact functor, meaning a
functor which preserves finite limits and colimits~\cite[1.1.4.1]{HA}.

More generally, given any $\i$-category $\C$ with finite limits,
the {\em stabilization} of $\C$ is the limit (in the
$\i$-category of $\i$-categories) of the tower
$$
\xymatrix{\cdots \ar[r]^{\Omega} & \C_*\ar[r]^{\Omega} & \C_*},
$$
where $\C_*$ denotes the pointed $\i$-category associated to $\C$ (the
full $\i$-subcategory of $\Fun(\Delta^1,\C)$ on those arrows whose
source is a final object $*$ of $\C$).
Provided $\C$ is presentable,
$\Stab(\C)$ comes equipped with a stabilization functor $\Sigma^\i_+
\colon \C\to\Stab(\C)$ functor from $\C$~\cite[1.4.4.4]{HA},
formally analogous to the suspension spectrum functor, and left
adjoint to the zero-space functor $\Omega^\i_-:\Stab(\C)\to\C$ (the
subscript indicates that we forget the basepoint).

If one works entirely in the world of presentable stable $\i$-categories and left adjoint functors thereof, then $\text{Stab}$ is left adjoint to the inclusion into the $\i$-category
of presentable $\i$-categories of the full $\i$-subcategory of
presentable stable $\i$-categories.
In other words, a morphism of presentable $\i$-categories $\C\to\D$ such
that $\D$ is stable factors (uniquely up to a contractible space of
choices) through the stabilization $\Sigma^\i_+\colon \C\to\text{Stab}(\C)$
of $\C$ (cf. \cite[1.4.4.4, 1.4.4.5]{HA}).
The $\i$-category $\Gpd_{\i/X}$ of spaces over a fixed space $X$ is presentable.

The discussion so far suggests two models for the $\i$-category of
spectra over $X$: one is $\Stab (\Gpd_{\i/X})$, and the other is $\Fun
(X^{\op}, \Stab (\Gpd_{\i}))$.  In fact these are equivalent: for any $\i$-groupoid $X$, the equivalence $\Gpd_{\i/X}\heq\Fun(X^{\op},\Gpd_\i)$
induces an equivalence
of stabilizations
$$
\Stab(\Gpd_{\i/X})\heq\Stab(\text{Fun}(X^{\op},\Gpd_\i)).
$$
Since limits in functor categories are computed pointwise, one easily
checks that
$$
\Stab(\Gpd_{\i/X})\heq\Stab(\Fun(X^{\op},\Gpd_\i))\heq\Fun(X^{\op},\Stab(\Gpd_\i)).
$$

\begin{Remark}\label{rem-may-sig}
May and Sigurdsson \cite{MR2271789} build a simplicial model category
$\mathcal{S}_{X}$ of orthogonal spectra parametrized by a topological
$X$.  In~\cite{ABG:tcu}, we prove that there is an equivalence of
$\i$-categories between the simplicial nerve of the May-Sigurdsson
category of parametrized orthogonal spectra
$\mathrm{N}(\mathcal{S}_X^{\cf})$ and the $\i$-category $\Fun(\Sing
X^{\op}, \Stab (\Gpd_{\i}))$.
\end{Remark}

\subsection{Parametrized $R$-modules and $R$-lines} \label{sec:glr-symm-mono}
We now fix an $A_\infty$-ring spectrum $R$.
Recall \cite[4.2.1.36]{HA} that there exists a presentable stable $\i$-category $\Mod{R}$ of (right) $R$-module spectra, and that this $\i$-category possesses a distinguished object $R$.

\begin{Definition} \label{inf-def-1}
An $R$-line is an $R$-module $M$ which admits an $R$-module
equivalence $M\heq R$.
\end{Definition}

Let $\Rwe$ denote the full $\i$-subgroupoid of $\Rmod$ spanned by the
$R$-lines.  This is not the same as the full $\i$-subcategory of
$\Rmod$ on the $R$-lines, as a map of $R$-lines is by definition an
equivalence.  We regard $\Rwe$ as a pointed $\i$-groupoid via the
distinguished object $R$.

\begin{Proposition}\label{prop:Rline=BAutR}
There is a canonical equivalence of $\infty$-groupoids
\[
B\Aut_R(R)\simeq\Rwe,
\]
and $\Aut_{R} (R)\simeq \GL{R}$ as monoidal $\infty$-groupoids.
\end{Proposition}

\begin{proof}
We regard $B\Aut_R(R)\subset\Mod{R}$ as the full subgroupoid of $\Mod{R}$ consisting of the single $R$-module $R$.
Hence $B\Aut_R(R)$ is naturally a full subgroupoid of $\Rwe$, and the fully faithful inclusion $B\Aut_R(R)\subset\Rwe$ is also essentially surjective by definition of $\Rwe$.
It is therefore an equivalence, so it only remains to show that $\GL{R}\simeq\Aut_R(R)$ as monoidal $\i$-groupoids.
This follows from the fact that $\End_R(R)\simeq\Omega^\infty(R)$, and $\Aut_R(R)\subset\End_R(R)$ is, by \ref{eqn:ho}, the monoidal subspace defined by the same condition as $\GL{R}\subset\Omega^\infty(R)$; namely, as the pullback
\[
\xymatrix{
\Aut_R(R)\ar[r]\ar[d] & \End_R(R)\ar[d]\\
\pi_0\End_R(R)^\times\ar[r] & \pi_0\End_R(R)\,,}
\]
where $\pi_0\End_R(R)^\times\cong\pi_0(R)^\times$ denotes the invertible homotopy classes of endomorphisms in the ordinary category $\Ho(\Mod{R})$.
\end{proof}

\begin{Definition}
Let $X$ be a space.
The $\i$-category of $R$-modules over $X$ is the $\i$-category $\Fun(X^{\op},\Mod{R})$ of presheaves of $R$-modules on $X$; similarly, the $\i$-category of $R$-lines over $X$ is the $\i$-category $\Fun(X^{\op},\Line{R})$ of presheaves of $R$-lines on $X$.
\end{Definition}

We will denote by $R_X$ the constant functor $X^{\op} \to 
\Line{R} \to \Mod{R}$ which has value $R$, and sometimes write
\[
\Mod{R_X} =\Fun(X^{\op},\Mod{R})
\]
for the $\i$-category of $R$-modules over $X$, and
\[
\Line{R_X} =\Fun(X^{\op},\Line{R})
\]
for the full $\i$-subgroupoid spanned by those $R$-modules over $X$ which factor
\[
X^{\op}\too\Line{R}\too\Mod{R}
\]
through the inclusion of the full $\i$-subgroupoid $\Line{R}\to\Mod{R}$.

\begin{Lemma}
The fiber over $X$ of the projection
$\Gpd_{\i/\Line{R}}\to\Gpd_\i$ is equivalent to the $\i$-groupoid
$\Line{R_X}$.
\end{Lemma}

\begin{proof}
$\Line{R_X}\heq\Fun(X^{\op},\Line{R})\heq\map_{\Gpd_\i}(X^{\op},\Line{R})$, and, in
general, the $\i$-groupoid $\map_{\C}(a,b)$ of maps from $a$ to $b$ in the
$\i$-category $\C$ may be calculated as the fiber over $a$ of the
projection $\C_{/b}\to\C$.
\end{proof}

\begin{Definition}
A {\em trivialization} of an $R_X$-module $L$ is an $R_X$-module
equivalence $L\to R_X$.  The $\i$-category $\Triv{R_X}$ of trivialized
$R$-lines is the slice category
\[
\Triv{R_X}\eqdef\Line{R_X}_{/R_X}.
\]
\end{Definition}

The objects of $\Triv{R_X}$ are {\em trivialized} $R_X$-lines, which
is to say $R_X$-lines $L$ with a trivialization $L\to R_X$; more
generally, an $n$-simplex $\Delta^n\to\Triv{R_X}$ of $\Triv{R_X}$ is a
map $\Delta^n\star\Delta^0\to\Line{R_X}$ of $\Line{R_X}$ which sends
$\Delta^0$ to $R_X$.  There is a canonical projection
$$
\iota_X:\Triv{R_X}\longrightarrow\Line{R_X}
$$
which sends the $n$-simplex $\Delta^n\star\Delta^0\to\Line{R_X}$ to
the $n$-simplex $\Delta^n\to\Delta^n\star\Delta^0\to\Line{R_X}$;
according to (the dual of) \HTT{Corollary 2.1.2.4}, this is a right
fibration, and hence a Kan
fibration as $\Line{R_X}$ is an $\i$-groupoid \HTT{Lemma 2.1.3.2}.

\begin{Lemma}\label{inf-t-pr-egl-bgl-fib}
Let $X$ be an $\i$-groupoid.  Then $\Triv{R_X}$ is a contractible
$\i$-groupoid, and the fiber, over a given $R_X$-line $f$, of the projection
$$
\iota_X \colon \Triv{R_X}\longrightarrow\Line{R_X}
$$
is the (possibly empty) $\i$-groupoid $\map_{\Line{R_{X}}} (f,R_{X})$.
\end{Lemma}

\begin{proof}
Once again, use the description of $\map_{\C}(a,b)$ as the fiber over $a$ of
the projection $\C_{/b}\to\C$, together with the fact that if $\C$ is
an $\i$-groupoid then $\C_{/b}$, an $\i$-groupoid with a final object,
is contractible.
\end{proof}

\begin{Corollary}~\label{cor:gl-kan-fib}
The Kan fibration
\[
\Aut_R(R)\to\Triv{R}\to\Line{R}
\]
is a simplicial model for the quasifibration $\GL{R}\to E\GL{R}\to
B\GL{R}$.
\end{Corollary}

\begin{proof}
By the preceding discussion, $\Triv{R}$ is a contractible Kan complex
and the projection $\Triv{R}\to\Line{R}$ is a Kan fibration.  By
Proposition \ref{prop:Rline=BAutR}, we have $\Aut_R(R)\simeq\GL{R}$.
\end{proof}

For $X$ the terminal Kan complex, we write $\Triv{R}$ in place of
$\Triv{R_X}$ and $\iota:\Triv{R}\to\Line{R}$ in place of
$\iota_X$.  Given $f:X\to\Line{R}$, we refer to a
factorization
\begin{equation} \label{eq:6}
\xymatrix{ & {\Rtriv}
 \ar[d]_\iota \\
{X^{\op}} \ar@{-->}[ur] \ar[r]_-{f} & {\Rwe.}  }
\end{equation}
of $f$ through $\iota$ as a {\em trivialization} of $f$.

\begin{Definition}
We write $\mathrm{Triv} (f)$ for the space of trivializations of $f$;
explicitly, it is the fiber over $f$ in the fibration 
\[
 \Fun(X^{\op},\Triv{R})\xra{\iota}\Fun(X^{\op},\Line{R}).
\]
\end{Definition}

\begin{Corollary}\label{triv(f)=triv(f^*L)}
There is a canonical equivalence of $\i$-groupoids
\[
\Fun(X^{\op},\Triv{R})\simeq\Triv{R_X}.
\]
Moreover, $\mathrm{Triv} (f)$ is equivalent to
$\map_{\Line{R_{X}}} (f,R_{X})$. 
\end{Corollary}

\begin{proof}
For the first claim, we have 
$$
\Fun(X^{\op},\Line{R}_{/R})\heq\Fun(X^{\op},\Line{R})_{/p^*R}\heq\Line{R_X}_{/R_X}.
$$
For the second, compare the two pull-back diagrams 
\[
\xymatrix{
{\map_{\Line{R_{X}}} (f,R_{X})}
 \ar[r] 
 \ar[d]
&
{\Line{R_{X}}_{/R_{X}}}
 \ar[d]
\\
{\{f \}} \ar[r]
&
{\Line{R_{X}}}
}
\]
and
\[
\xymatrix{
{\mathrm{Triv} (f)}
 \ar[r] 
 \ar[d]
&
{\Fun (X^{\op},\Triv{R})}
 \ar[d]
\\
{\{f \}} 
\ar[r]
&
{\Fun (X^{\op},\Line{R})},
}
\]
in which the two right-hand fibrations are equivalent.
\end{proof}

%

A map of spaces $f:X\to Y$ gives rise to a restriction functor
$$
f^*:\Mod{R_Y}\to\Mod{R_X}
$$
which admits a right adjoint $f_*$ as well as a left adjoint $f_!$.
This means that, given an $R_X$-module $L$ and an $R_Y$-module $M$,
there are natural equivalences of $\i$-groupoids
$$
\map(f_!L,M)\heq\map(L,f^*M)
$$
and
$$
\map(f^*M,L)\heq\map(M,f_*L).
$$

An important point about these functors is the following.   

\begin{Proposition}\label{t-triv-is-smash}
Let $\pi: X \to \ptspace$ be the projection to a point and let
$\pi^{*}: \Rmod \to 
\Mod{R_{X}}$ be the resulting functor. If $M$ is an $R$-module, then 
\begin{equation}\label{eq:7}
    \pi_{!}\pi^{*}M \heq \splus X \Smash M.
\end{equation}
\end{Proposition}

\begin{proof}
We use the equivalence $\Mod{R_{X}}\heq \Fun (X^{\op},\Rmod)$, and compute
in $\Fun (X^{\op},\Rmod).$  In that case the the left hand
side in \eqref{eq:7} is the colimit of the constant map of $\i$-categories
\[
    X^{\op} \xra{M} \Rmod.
\]
This map is equivalent to the composition 
\[
    X^{\op} \overset{1}{\too} \Gpd_{\i} \xra{\splus} \Stab (\Gpd_{\i})
    \xra{(\slot)\Smash M} \Rmod.
\]
The second two functors in this composition commute with colimits, and
equation \ref{colim1=X} says that $X\heq \colim (1: X^{\op}\xra{}\Gpd_{\i}).$
\end{proof}

\subsection{Thom spectra}
\label{sec:thom-r-modules}
We continue to fix an $A_\infty$-ring spectrum $R$.



\begin{Definition}
The Thom $R$-module spectrum is the functor
$$
M:\Gpd_{\i/\Line{R}}\longrightarrow\Rmod
$$
which sends $f:X^{\op}\to\Line{R}$ to the colimit of the composite 
\[ X^{\op}\overset{f}\to\Line{R}\overset{i}{\to}\Mod{R}. \]
Equivalently $Mf$ is the left
Kan extension
$$
Mf\eqdef p_!(i\circ f)
$$
along the map $p: X^{\op}\to \ptspace$. 
\end{Definition}




\begin{Proposition}\label{quotient}
Let $G$ be an $\i$-group (a group-like monoidal $\i$-groupoid) with classifying space $BG$ and suppose given a map $f:BG\to\Line{R}$.  Then
$$
Mf\heq R/G,
$$
where $G$ acts on $R$ via the map $\Omega f:G\heq\Omega
BG\to\Omega(\Line{R})\heq\Aut_R(R)$.
\end{Proposition}

\begin{proof}
Both $Mf$ and $R/G$ are equivalent to the colimit of the composite functor $BG^{\op}\to B\Aut_R(R)\heq\Line{R}\to\Mod{R}$.
\end{proof}

\subsection{Orientations}
With these in place, one can analyze the space of orientations in a
straightforward manner, as follows.  First of all observe that, by
definition, we have an equivalence 
\[
   \map_{\Mod{R}}(Mf,R) \heq \map_{\Mod{R_{X}}}(f,p^{*}R).
\]

\begin{Definition} \label{inf-def-2}
The space of \emph{orientations} of $Mf$ is the pullback
\begin{equation}\label{inf-eq:2}
\xymatrix{
\Rorient (Mf,R) \ar[r] \ar[d]^-{\heq} & \map_{\Rmod}(Mf,R) \ar[d]^-{\heq} \\
\map_{\Line{R_{X}}}(f,p^*R) \ar[r] & \map_{\Mod{R_{X}}}(f,p^*R).\\
}
\end{equation}
\end{Definition}

The $\i$-groupoid $\Rorient (Mf,R)$ enjoys an obstruction theory
analogous to that of the space of orientations described in
\cite{units-sma}.  The following theorem is the analogue in this
context of~\cite[3.20]{units-sma}.

\begin{Theorem} \label{t-th-or-thy-infty-lifting}
Let $f:X^{\op}\to\Rwe$ be a map, with associated Thom $R$-module $Mf$.
Then the space of orientations $Mf\to R$ is equivalent to
the space of lifts in the diagram
\begin{equation} \label{eq:26}
\xymatrix{ & {\Rtriv}
 \ar[d]_\iota \\
X^{\op} \ar@{-->}[ur] \ar[r]_-{f} & {\Rwe.}  }
\end{equation}
\end{Theorem}

\begin{proof}
Corollary \ref{triv(f)=triv(f^*L)} says that the space
$\mathrm{Triv}(f)$ of factorizations of $f$ through $\iota$ is
equivalent to the mapping space $\map_{\Line{R_{X}}} (f,p^{*}R)$.  
\end{proof}

\begin{Corollary}\label{t-thom-iso-infty-cat-I}
An orientation of $Mf$ determines an equivalence of $R$-modules
\[
Mf \heq \splus X \Smash R.
\]
\end{Corollary}

\begin{proof}
If $\C$ is an $\infty$-category,  write $\mathrm{Iso} (\C) (a,b)$
for the subspace $\map_{\C^{\simeq}} (a,b)\subseteq \map_{\C} (a,b)$
consisting of equivalences  (see \eqref{eqn:ho}).  By definition, 
$\map_{\Line{R_X}}(f,R_{X}) =\mathrm{Iso} (\Mod{R_{X}}) (f,p^*R)$,
and so \eqref{inf-eq:2} gives an equivalence 
$\Rorient(Mf,R)\too \mathrm{Iso} (\Mod{R_{X}})(f,p^*R)$.  The desired map is
the composite 
\begin{align*}
\mathrm{Iso} (\Mod{R_{X}}) (f,p^*R)\too
\mathrm{Iso} (\Mod{R}) (p_!f,p_!p^*R)\too \mathrm{Iso} (\Mod{R})
(p_!f,\Sigma^\i_+ X\land R).
\end{align*}
Here the second
map applies $p_!$ and the last map composes with the
equivalence $p_!p^*R\to\Sigma^\i_+ X\land R$ of Proposition \ref{t-triv-is-smash}.
\end{proof}

\subsection{Twisted homology and cohomology}

Recall that the $R$-module Thom spectrum $Mf$ of the map
$f\colon X^{\op}\to\Rwe$, which 
we think of as classifying an $R$-line bundle on $X$, is the
pushforward $Mf\simeq p_! f$ of the composite
\[
X^{\op}\overset{f}{\too}\Rwe\too\Mod{R}.
\]
The homotopy groups $\pi_n Mf$ can be computed as homotopy classes of
$R$-module maps from $\Sigma^n R$ to $Mf$, which is a convenient
formulation because the twisted $R$-cohomology groups are dually
homotopy classes of $R$-module maps from $Mf$ to $\Sigma^n R$. 

\begin{Definition}
Let $R$ be an $A_\i$ ring spectrum, let $X$ be a space with projection
$p\colon X\to *$ to the point, and let $f\colon X\to\Rwe$ be an
$R$-line bundle on $X$.  Then the $f$-twisted $R$-homology and
$R$-cohomology of $X$ are the mapping spectra 
\begin{align*}
R^f(X) =&\,\Map_R(R,Mf)\simeq Mf\\
R_f(X) =&\,\Map_R(Mf,R)\simeq\Map_{R_X}(f,R_X),
\end{align*}
formed in the stable $\i$-category $\Mod{R}$ of $R$-modules (or
$\Fun(X^{\op},\Mod{R})$ of $R_X$-modules). 
\end{Definition}

\noindent Here recall that $R_X\simeq p^*R$ is the constant bundle of
$R$-modules $X^{\op} \to \Rwe\to\Mod{R}$, and the equivalence of
mapping spectra $\Map_R(Mf,R)\simeq\Map_{R_X}(f,R_X)$ follows from the
equivalence, for each integer $n$, of mapping spaces 
\[
\map_R(p_! f,\Sigma^n R)\simeq\map_{R_X}(f,p^*\Sigma^n R)
\]
that results from the fact that $p^*$ is right adjoint to $p_!$.

Note that, since $R$ is only assumed to be an $A_\infty$ ring
spectrum, the homotopy category of $\Mod{R}$ does not usually admit a
closed monoidal structure with unit $R$;  
nevertheless, we still regard $R_f(X)$ as the ``$R$-dual'' spectrum
$\Map_R(Mf,R)$ of $Mf\simeq R^f(X)$, or as the ``spectrum of (global)
sections'' $\Map_{R_X}(f,R_X)$ of the $R$-line bundle $f$. 
Also, the notation $R^f(X)$ and $R_f(X)$ is designed so that, for an
integer $n$, we have the $f$-twisted $R$-homology and $R$-cohomology
{\em groups} 
\begin{align*}
R^f_n(X) =&\,\pi_0\map_R(\Sigma^n R,Mf)\cong\pi_n Mf\\
R_f^n(X) =&\,\pi_{0}\map_R(Mf,\Sigma^n
R)\cong\pi_{0}\map_{R_X}(f,p^*\Sigma^n R). 
\end{align*}

A consequence of our work with orientations is the following
untwisting result:

\begin{Corollary}
If $f\colon X^{\op}\to\Rwe$ admits an orientation, then $Mf\simeq
R\land\Sigma^\infty_+ X$, and the twisted $R$-homology and
$R$-cohomology spectra 
\begin{align*}
R^f(X)\simeq &\, R\land\Sigma^\infty_+ X\\
R_f(X)\simeq &\, \Map(\Sigma^\infty_+ X,R)
\end{align*}
reduce to the ordinary $R$-homology and $R$-cohomology spectra of $X$.
\end{Corollary}

\begin{proof}
Indeed, Corollary \ref{t-thom-iso-infty-cat-I} gives equivalences
$\Map_R(R,Mf)\simeq Mf\simeq R\land\Sigma^\infty_+ X$ and
$\Map_R(R\land\Sigma^\i_+ X,R)\simeq\Map(\Sigma^\i_+ X,R)$. 
\end{proof}

\section{Morita theory and Thom spectra}\label{sec:comp-thom-spectra}

In this section we interpret the construction of the Thom spectrum
from the perspective of Morita theory.  This viewpoint is implicit in
the ``algebraic'' definition of the Thom spectrum of $f \colon X \to
B\GL{R}$ in~\cite{units-sma} as the derived smash product 
\[
 \Malg f\eqdef \splus P \sma^\mathrm{L}_{\splus \GL{R}} R,
\]
where $P$ is the pullback of the diagram
\[
\xymatrix{
X \ar[r] & B\GL{R} & \ar[l] E\GL{R}. \\
}
\]
As passage to the pullback induces an equivalence between spaces over
$B\GL{R}$ and $\GL{R}$-spaces, and the target category of $R$-modules
is stable, we can regard the Thom spectrum as essentially given by a
functor from (right) $\splus \GL{R}$-modules to $R$-modules. 

Roughly speaking, Morita theory (more precisely, the
Eilenberg-Watts theorem) implies that any continuous functor from 
(right) $\splus \GL{R}$-modules to (right) $R$-modules which preserves
homotopy colimits and takes $\GL{R}$ to $R$ can be realized as
tensoring with an appropriate $(\splus \GL{R})$-$R$ bimodule.  In
particular, this tells us that the Thom spectrum functor is
characterized amongst such functors by the additional data of the
action of $\GL{R}$ on $R$, equivalently a map $B\GL{R}\to B\GL{R}.$

Beyond its conceptual appeal, this viewpoint on the Thom spectrum
functor provides the basic framework for comparing the construction
which we have discussed in this paper with $\Malg$ and also with
the ``neo-classical'' construction of Lewis and May and the
parametrized construction of May and Sigurdsson.

After discussing the analogue of the classical Eilenberg-Watts theorem
in the context of ring spectra in \S\ref{sec:morita}, in
\S\ref{sec:colim-pres-funct} we classify 
colimit-preserving functors between $\i$-categories.  Our
classification leads in \S\ref{sec:i-categorical-thom} to a
characterization of the ``geometric'' Thom spectrum functor $M=\Mgeo$
of this paper, which 
serves as the basis for comparison with the ``algebraic'' Thom
spectrum $\Malg$ from \cite{units-sma}.  

In \S\ref{sec:revi-algebr-thom} we briefly review the
construction of $\Malg$, and
characterize it using Morita theory.  In
\S\ref{sec:comparison} we prove the equivalence of $\Mgeo$
and $\Malg$.  
The close relationship between our $\i$-categorical
construction of the Thom spectrum and the definition of May and
Sigurdsson \cite[23.7.1,23.7.4]{MR2271789} allows us (in
\S\ref{sec:neo-classical-thom}) to compare May and Sigurdsson's construction of the Thom spectrum 
(and by extension the ``neo-classical'' Lewis-May construction) to the
ones in this paper.

In \S\ref{sec:algebr-thom-spectr} we also sketch a direct
comparison between $\Mgeo$ and $\Malg$;
although the argument  does not characterize the functor among all functors from
$\GL{R}$-modules to $R$-modules, we believe it provides a useful
concrete depiction of the situation.  

\subsection{The Eilenberg-Watts theorem for categories of module spectra}
\label{sec:morita}

The key underpinning of classical Morita theory is the Eilenberg-Watts
theorem, which for rings $A$ and $B$ establishes an equivalence
between the category of colimit-preserving functors $\Mod{A} \to
\Mod{B}$ and the category of $(A,B)$-bimodules.  The proof of the theorem
proceeds by observing that any functor $T\colon \Mod{A} \to \Mod{B}$
specifies a bimodule structure on $TA$ with the $A$-action given by
the composite 
\[A \to F_A (A,A) \to F_B(TA, TA).\]
It is then straightforward to check that the functor $- \otimes_A TA$
is isomorphic to the functor $T$, using the fact that both of these
functors preserve colimits.

In this section, we discuss the generalization of this result to the
setting of categories of module spectra.  The situation here is more
complicated than in the discrete case; for instance, it is
well-known that there are equivalences between categories of module
spectra which are not given by tensoring with bimodules, and there are
similar difficulties with the most general possible formulation of the
Eilenberg-Watts theorem.  However, much of the subtlety here comes
from the fact that unlike in the classical situation, compatibility
with the enrichment in spectra is not automatic (see for example the
excellent recent paper of Johnson \cite{Niles} for a comprehensive
discussion of the situation).  By assuming our functors are enriched,
we can recover a close analogue of the classical result.

Let $A$ and $B$ be (cofibrant) $S$-algebras, and let $T$ be an
enriched functor
\[
T\colon \Mod{A} \to \Mod{B}.  
\]
Specifically, we
assume that $T$ induces a map of function spectra $F_A(X,Y) \to
F_B(TX,TY)$, and furthermore that $T$ preserves tensors (in
particular, homotopies) and homotopy colimits.
For instance, these conditions are satisfied if $T$ is a Quillen
left-adjoint.  The assumption that $T$ is homotopy-preserving implies that $T$ preserves weak equivalences between
cofibrant objects and so admits a total left-derived functor $T^\mathrm{L}
\colon \Ho{\Mod{A}} \to \Ho{\Mod{B}}$.  Furthermore, $T(A)$ is an
$A$-$B$ bimodule with the bimodule structure induced just as above.

Using an elaboration of
the arguments of \cite[4.1.2]{MR1928647} (see also
\cite[4.20]{MR2122154}) we now can prove the following Eilenberg-Watts
theorem in this setting.  We will work in the EKMM categories of
$S$-modules \cite{EKMM}, so we can assume that all objects are fibrant.

\begin{Proposition}\label{prop:watts}
Given the hypotheses of the preceding discussion, there is a natural
isomorphism in $\Ho{\Mod{B}}$ between the total left-derived functor
$T^\mathrm{L}(-)$ and the derived smash product $(-) \sma^\mathrm{L} T(A)$, regarding
$T(A)$ as a bimodule as above.
\end{Proposition}

\begin{proof}
By continuity, there is a natural map of $B$-modules
\[(-) \sma_A T(A) \to T(-).\]
Let $T'$ denote a cofibrant replacement of $T(A)$ as an $A$-$B$
bimodule.  Since the functor $(-) \sma_A T'$ preserves weak
equivalences between cofibrant $A$-modules, there is a total
left-derived functor $(-) \sma_A^\mathrm{L} T'$ which models $(-) \sma_A^\mathrm{L}
T(A)$.  Thus, the composite
\[
(-) \sma_A T' \to (-) \sma_A T(A) \to T(-).
\]
descends to the homotopy category to produce a natural map
\[
(-) \sma_A^\mathrm{L} T(A) \to T^\mathrm{L}(-).
\]
The map is clearly an equivalence for the free $A$-module of rank
one; i.e. $A$.  Since both sides commute with homotopy colimits, we
can inductively deduce that the first map is an equivalence for all
cofibrant $A$-modules, and this implies that the map of derived
functors is an isomorphism.
\end{proof}

To characterize the
Thom spectrum functor amongst functors from spaces over ${B\GL{R}}$ to
$R$-modules, it is useful to formulate Proposition~\ref{prop:watts} in
terms of $\i$-categories.  One reason is that (as we recall in 
Subsection~\ref{sec:revi-algebr-thom}) the ``algebraic'' Thom spectrum
of \cite{units-sma} is  
the composition of a right derived
functor (which is an equivalence) and a left derived functor.  We
remark that much of the technical difficulty in the neo-classical
theory of the Thom spectrum functor arises from the difficulties
involved in dealing with point-set models of such composites.
This is the kind of formal situation that the
$\i$-category framework handles well.  

\subsection{Colimit-preserving functors}

\label{sec:colim-pres-funct}

In this section we study functors between $\i$-categories
which preserve colimits.  Specializing to module categories, we obtain
a version of the Eilenberg-Watts theorem which applies to both
the algebraic and the geometric Thom spectrum.  

We begin by considering cocomplete $\i$-categories.  Let $\C$ be a
small $\i$-category, and consider the $\i$-topos
$\Pre(\C)=\text{Fun}(\C^{\op},\spaces)$ of presheaves of 
$\i$-groupoids on $\C$.  Recall that $\Pre(\C)$ comes equipped with a
fully faithful Yoneda embedding 
\begin{equation} \label{eq:63}
\C\longrightarrow\Pre(\C)
\end{equation}
which exhibits $\Pre(\C)$ as the ``free cocompletion'' \cite[5.1.5.8]{HTT} of $\C$.
More precisely, writing $\Fun^\mathrm{L}(\C,\D)$ for the full subcategory of $\Fun(\C,\D)$ consisting of the colimit-preserving functors, we have the following:

\begin{Lemma}[\HTT{5.1.5.6}] \label{t-le-yoneda-ini-colim-preserv}
For any cocomplete $\i$-category $\D$,
precomposition with the Yoneda embedding induces an equivalence of
$\i$-categories
\begin{equation}\label{eq:74}
\Fun^\mathrm{L}(\Pre(\C),\D)\longrightarrow\Fun(\C,\D).
\end{equation}
\end{Lemma}

We shall be particularly interested in the case that $\C$ is an
$\i$-groupoid, so that
\begin{equation}\label{eq:8}
    \Pre (\C) = \Fun (\C^{\op},\Gpd_{\i}) \heq \Gpd_{\i/\C},
\end{equation}
as in Remark \ref{rem-4}.
In particular, given a functor $f:\C\to\D$, we may extend by colimits to a colimit-preserving functor $\tilde{f}\colon\Gpd_{\i/\C}\to\D$.

\begin{Corollary}\label{t-co-yoneda-inverse}
If $g\colon \Gpd_{\i/\C}\to \D$ is any colimit-preserving functor  whose
restriction along the Yoneda embedding $\C\to\Gpd_{\i/\C}$ is equivalent to $f$, then $g$ is equivalent to $\tilde{f}.$ 
\end{Corollary}

\begin{Lemma}[\HA{1.4.4.4, 1.4.4.5}] Let $\C$ and $\D$ be
presentable $\i$-categories such that $\D$ is stable.  Then
$$
\Omega^\i_-:\Stab(\C)\longrightarrow\C
$$
admits a left adjoint
$$
\Sigma^\i_+:\C\longrightarrow\Stab(\C),
$$
and precomposition with the $\Sigma^\i_+$ induces an equivalence of
$\i$-categories
$$
\Fun^\mathrm{L}(\Stab(\C),\D)\longrightarrow\Fun^\mathrm{L}(\C,\D).
$$
\end{Lemma}

Combining the universal properties of stabilization and the Yoneda
embedding, we obtain the following equivalence of $\i$-categories.

\begin{Corollary}
Let $\C$ and $\D$ be $\i$-categories such that $\D$ is stable and
presentable.  Then there are equivalences of $\i$-categories
$$
\Fun^\mathrm{L}(\Stab(\Pre(\C)),\D)\heq\Fun^\mathrm{L}(\Pre(\C),\D)\heq\Fun(\C,\D).
$$
\end{Corollary}

\begin{proof}
This follows from the last two lemmas.
\end{proof}

Now suppose that $\C$ and $\D$ have distinguished objects, given by
maps $*\to\C$ and $*\to\D$ from the trivial $\i$-category $*$.  Then
$\Pre(\C)$ and $\Stab(\Pre(\C))$ inherit distinguished objects via the
composite
$$
*\longrightarrow\C\overset{i}{\longrightarrow}\Pre(\C)\overset{\Sigma^\i_{+}}{\longrightarrow}\Stab(\Pre(\C)),
$$
where $i$ denotes the Yoneda embedding.  Note that the fiber sequence
$$
\Fun_{*/}(\C,\D)\longrightarrow\Fun(\C,\D)\longrightarrow\Fun(*,\D)\heq\D
$$
shows that the $\i$-category of pointed functors is equivalent to the
fiber of the evaluation map $\Fun(\C,\D)\to\D$ over the distinguished
object of $\D$.

\begin{Proposition}
Let $\C$ and $\D$ be $\i$-categories with distinguished objects such
that $\D$ is stable and cocomplete.  Then there are equivalences of
$\i$-categories
$$
\Fun_{*/}^\mathrm{L}(\Stab(\Pre(\C)),\D)\heq\Fun_{*/}^\mathrm{L}(\Pre(\C),\D)\heq\Fun_{*/}(\C,\D).
$$
\end{Proposition}

\begin{proof}
Take the fiber of
$\Fun^\mathrm{L}(\Stab(\Pre(\C)),\D)\heq\Fun^\mathrm{L}(\Pre(\C),\D)\heq\Fun(\C,\D)$
over $*\to\D$.
\end{proof}

\begin{Corollary}
Let $G$ be a group-like monoidal $\i$-groupoid $G$, let $BG$ be a
one-object $\i$-groupoid with $G\heq\Aut_{BG}(*)$, and let $\D$ be a
stable and cocomplete $\i$-category with a distinguished object $*$.
Then
\begin{align*}
\Fun_{*/}^\mathrm{L}(\Stab(\Pre(BG)),\D)\heq&\Fun_{*/}^\mathrm{L}(\Pre(BG),\D)\heq\\
&\Fun_{*/}(BG,\D)\heq\Fun(BG,B\Aut_\D(*));
\end{align*}
that is, specifying an action of $G$ on the distinguished object $*$
of $\D$ is equivalent to specifying a pointed colimit-preserving
functor from $\Pre(BG)$ (or its stabilization) to $\D$.
\end{Corollary}

\begin{proof}
A base-point preserving functor $BG\to\D$ necessarily factors through the full subgroupoid $B\Aut_\D(*)$.
\end{proof}

Note that the $\i$-category $\Fun(BG,B\Aut_\D(*))$ is actually an
$\i$-groupoid, as $B\Aut_\D(*)$ is an $\i$-groupoid.

Putting this all together, consider the case in which the target
$\i$-category $\D$ is the $\i$-category of right $R$-modules for an
associative $S$-algebra $R$, pointed by the free rank one $R$-module
$R$.  Then $\Aut_\D(*)\heq GL_1 R$, and we have an 
$\i$-categorical version of the Eilenberg-Watts theorem.

\begin{Corollary}\label{t-co-e-w-i-sp-ov-BG} 
The space of pointed
colimit-preserving maps from the $\i$-category of spaces over $BG$ to
the $\i$-category of $R$-modules is 
equivalent to the space of monoidal maps from $G$ to $\GL{R}$, or
equivalently the space of maps from $BG$ to $B\GL{R}$.
\end{Corollary} 

\subsection{$\i$-categorical Thom spectra, revisited}
\label{sec:i-categorical-thom}

We now return to the definition of Thom spectra from
\S\ref{sec:units-via-infty} and interpret that construction in light
of the work of the previous subsections.  To avoid confusion with the
Thom spectrum constructed in \cite{units-sma}, in this section we
write $\Mgeo$ for the Thom spectrum of \S\ref{sec:units-via-infty}.

Let $R$ be an algebra in $\Stab (\Gpd_{\i})$,
and form  the $\i$-categories $\Rmod$ and $\Rwe.$    Given a map of
$\i$-groupoids
\[
    f: X \to \Rwe,  
\]
the ``geometric'' Thom spectrum we constructed in \S\ref{sec:units-via-infty} is the push-forward of the restriction to $X$ of the
tautological $R$-line bundle $\mathrm{id}_{\Rwe}$, the identity of $\Line{R}$. More precisely, $\Mgeo f \heq \colim (f: X \to \Rwe
\xra{}\Rmod)$, and in 
particular, $\Mgeo$ preserves ($\i$-categorical) colimits.  

\begin{Proposition} \label{t-pr-yoneda-taut}
The restriction of $\Mgeo:\Gpd_{\i/\Line{R}}\!\!\to \Rmod$ along the Yoneda
embedding 
$$
\Line{R}\longrightarrow\Fun(\Line{R}^{\op},\Gpd_\i)\heq\Gpd_{\i/\Line{R}}
$$
is equivalent to the inclusion $\Line{R}\longrightarrow\Mod{R}$ of the full $\i$-subgroupoid on $R$.
\end{Proposition}

\begin{proof}
Consider the colimit-preserving functor $\Gpd_{\i/\Line{R}}\to\Mod{R}$ induced by the canonical inclusion $\Line{R}\to\Mod{R}$.
As we explain in Corollary \ref{t-co-yoneda-inverse}, it sends
$X\to\Line{R}$ to the colimit of the composite
$X\to\Line{R}\to\Mod{R}$.
\end{proof}

Together with Corollary \ref{t-co-yoneda-inverse}, the Proposition
implies the following.

\begin{Corollary}[Proposition~\ref{t-pr-axiom-thom-intro}] \label{t-co-axiom-geom-thom-functor}
A functor $\Gpd_{\i/\Line{R}}\to\Mod{R}$ is equivalent to $\Mgeo$ if
and only if it preserves colimits and its restriction along the Yoneda
embedding
$\Line{R}\to\Fun(\Line{R}^{\op},\Gpd_\i)\heq\Gpd_{\i/\Line{R}}$ is
equivalent to the inclusion of
$\Line{R}$ into $\Mod{R}$. 
\end{Corollary}

\subsection{A review of the algebraic Thom spectrum functor }
\label{sec:revi-algebr-thom}

We briefly recall the 
``algebraic'' construction  of the Thom spectrum from \cite{units-sma}.
For an $\ainfty$ ring spectrum $R$, the classical construction yields $\GL{R}$ as
an $\ainfty$ space.  This means we expect to be able to form
constructions $B \GL{R}$ and $E \GL{R}$, and so given a classifying
map $f \colon X \to B \GL{R}$ obtain a $\GL{R}$-space $P$ as the
pullback of the diagram 
\[
\xymatrix{
X \ar[r] & B\GL{R} & \ar[l] E\GL{R}. \\
}
\]
We then define the Thom spectrum associated to $f$ as the derived
smash product 
\begin{equation}\label{eq:9}
\Malg f\eqdef \splus P \sma^{L}_{\splus \GL{R}} R,
\end{equation}
where $R$ is the $\splus \GL{R}$-$R$ bimodule specified by the
canonical action of $\splus \GL{R}$ on $R$. 

In order to make this outline precise, the companion paper used the
technology of $*$-modules \cite{Blumberg-thesis,
  Blumberg-Cohen-Schlichtkrull}, which are a symmetric monoidal model for
the category of spaces such that monoids are precisely $A_\infty$
spaces and commutative monoids are precisely $E_\infty$ spaces.
Denote the category of $*$-modules by $\aM_*$.  As an $A_\infty$ (or
$E_\infty$) space, $\GL{R}$ gives rise to a monoid in the the category
of $*$-modules.  We will abusively continue to use the notation
$\GL{R}$ to denote a model of $\GL{R}$ which is cofibrant as a monoid
in $*$-modules.  We can compute $B\GL{R}$ and $E\GL{R}$ as two-sided
bar constructions with respect to the symmetric monoidal product
$\boxtimes$:
\[
E_{\boxtimes}\GL{R} = B_\boxtimes(*,\GL{R},\GL{R}) \qquad \textrm{and}\qquad
B_{\boxtimes}\GL{R} = B_\boxtimes(*,\GL{R},*).
\]
The map $E_{\boxtimes}\GL{R} \to B_{\boxtimes}\GL{R}$ models the universal
quasifibration~\cite[3.8]{units-sma}.  Furthermore, there 
is a homotopically well-behaved category $\aM_{\GL{R}}$ of
$\GL{R}$-modules in $\aM_*$~\cite[3.6]{units-sma}.

Now, given a fibration of $*$-modules $f \colon X \to B_{\boxtimes}\GL{R}$, we take the pullback of the diagram
\[
\xymatrix{
X \ar[r] & B_{\boxtimes}\GL{R} & \ar[l] E_{\boxtimes}\GL{R} \\
}
\]
to obtain a $\GL{R}$-module $P$.  This procedure defines a functor
from $*$-modules over $B_{\boxtimes}\GL{R}$ to $\GL{R}$-modules; since
we are assuming $f$ is a fibration, we are computing the derived
functor.  Applying $\Sigma^{\infty}_{\bL+}$, we obtain a right
$\Sigma^{\infty}_{\bL+} \GL{R}$-module $\Sigma^{\infty}_{\bL+} P$, and
so we can define $\Malg f$ as above.  (Here $\Sigma^{\infty}_{\bL+}$
is the appropriate model of $\splus$ in this setting.)

The functor which sends $f$ to $P$ induces an equivalence of
$\i$-categories
\[
    \N(({\EKMM_{*/B_{\boxtimes}{\GL{R}}}})^{\cf}) \heq
    \N(({\EKMM_{\GL{R}}})^{\cf}), 
\]
as a consequence of~\cite[3.19]{units-sma}.  Together with Proposition
\ref{prop:watts}, this gives a characterization of the algebraic Thom
spectrum functor.

\begin{Proposition}\label{abstract-alg-thom}
Let
\[
T\colon \aM_{\GL{R}} \to \aM_{R}
\]
be a continuous, colimit-preserving
functor which sends $\GL{R}$ to an $R$-module $R'$ homotopy equivalent
to $R$ in such a way that 
\[
\GL{R}\simeq\mathrm{End}_{\aM_{\GL{R}}}(\GL{R})\longrightarrow\mathrm{End}_{\aM_{R}}(R')\simeq\mathrm{End}_{\aM_{R}}(R) 
\]
is homotopy equivalent to the inclusion
$\GL{R}\simeq\mathrm{Aut}(R)\to\mathrm{End}(R)$.
(Here $\Aut$ and $\End$ refer to the derived automorphism and
endomorphism spaces respectively.)
Then $T^\mathbb{L}$, the left-derived functor of $T$, is homotopy
equivalent to
\[
\Sigma^{\infty}_{\bL+} (-)\land^\mathbb{L}_{\Sigma^\infty_{\bL+}
  \GL{R}} R \colon \aM_{\GL{R}}\to \aM_R.
\]    
\end{Proposition}

\begin{proof}
The stability of $\Mod{R}$ and Proposition~\ref{prop:watts} together
imply that $T^\mathbb{L}$ is homotopy equivalent to
$\Sigma^\i_{\bL+}(-)\land^\mathbb{L}_{\Sigma^\i_{\bL+}\GL{R}} B$ for some
$(\Sigma^\i_{\bL+}\GL{R},R)$-bimodule $B$. 
Since $T(\Sigma^\i_{\bL+}\GL{R})\simeq R$, we must have $B\simeq R$; since
the left action of $\GL{R}$ on itself induces (via the equivalence
$R'\simeq R$) the canonical action of $\Sigma^{\infty}_{\bL+} \GL{R}$
on $R$, we conclude that $B\simeq R$ as $(\Sigma^\i_{\bL+}\GL{R},R)$-bimodules. 
\end{proof}

\subsection{Comparing notions of Thom spectrum}
\label{sec:comparison}

In this section, we show that, on underlying $\i$-categories, the
algebraic Thom $R$-module functor is equivalent to the
geometric Thom spectrum functor via the characterization of
Corollary~\ref{t-co-axiom-geom-thom-functor}.  

Let $\EKMM_{S}$ be the category of EKMM $S$-modules \cite{EKMM}.  According to the 
discussion in~\cite[\S 1.4.3]{HA} (and using the comparisons
of~\cite{MR1806878}), there is an equivalence of $\i$-categories 
\begin{equation} \label{eq:65}
     \N \EKMM_{S}^{\cf} \heq \Stab (\Gpd_{\i})
\end{equation}
which induces equivalences of $\i$-categories of algebras and
commutative algebras
\begin{equation} \label{eq:64}
     \N\Alg (\EKMM_{S})^{\cf} \heq \Alg (\Stab (\Gpd_{\i})) \qquad 
     \N\CAlg(\EKMM_{S})^{\cf} \heq \CAlg (\Stab (\Gpd_{\i})).
\end{equation}
Let $R$ be a cofibrant-fibrant EKMM $S$-algebra, and let $R'$ be
the corresponding algebra in $\Alg (\Stab (\Gpd_{\i}))$.  The
equivalence \eqref{eq:65} induces an equivalence of $\i$-categories
\begin{equation} \label{eq:66}
   \N(\aM_{R}^{\cf}) \heq \Mod{R'}.
\end{equation}
Proposition \ref{cor:gl-kan-fib} gives an equivalence of
$\i$-groupoids
\begin{equation}\label{eq:67}
   \Sing{B\GL{R}} \heq     \N ({(\Line{R})}^{\cf})
\end{equation}
and so putting \eqref{eq:66} and \eqref{eq:67} together with the
comparisons of~\cite[3.7]{units-sma} we have equivalences of
$\i$-categories 
\[
     \N((\aM_{*/\B_{\boxtimes}\GL{R}})^{\cf}) \heq \N
     ({(\Top_{/B\GL{R}})}^{\cf}) \heq \Gpd_{\i/\Line{R'}}. 
\]

\begin{Proposition} \label{t-pr-thom-spectra-equivalent}
The functor 
\[
     \Gpd_{\i/\Mod{R'}}\heq \N({(\Top_{/B\GL{R}})}^{\cf}) \xra{\N \Malg}
                             \N(\aM_{R}^{\cf}) \heq \Mod{R'},
\]
obtained by passing the Thom $R$-module functor $\Malg$ of
\cite{units-sma} though the indicated equivalences, is
equivalent to the Thom $R'$-module functor of
\S\ref{sec:units-via-infty}.
\end{Proposition}

\begin{proof}
Let $\C$ denote the topological category with a single object $*$ and
\[
\map_\C(*,*)=\GL{R}=\Aut_R(R^\cf)\simeq\Aut_{R'}(R').
\]
Note that $\C$ is naturally a topological subcategory of
$\aM_{\GL{R}}$ (the full topological subcategory on $\GL{R}$) and by
definition a topological subcategory of $\aM_{R}$.  Note also that 
\[
   \N\C \simeq B\Aut (R')\simeq \Line{R'}.
\]
As in Proposition~\ref{abstract-alg-thom}, the continuous functor
$$
T^\mathrm{L}\colon \aM_{\GL{R}}\longrightarrow \aM_{R}
$$
determined by $\Malg$ 
has the property that its restriction to $\C$ is equivalent to the
inclusion of the topological subcategory $\C\to\aM_{R}$.  Taking
simplicial nerves, and recalling that 
\[
     \N (\aM_{\GL{R}}^{\cf}) \heq  \N ((\Top_{/B\GL{R}})^{\cf}) \heq
      \Fun
     (\N\C^{\op},\Gpd_{\i}),
\]
we see that 
\[
\N(T^\mathrm{L})\colon
\Fun(\N\C^{\op},\Gpd_\i)\simeq\N(\aM_{\GL{R}}^\cf)\longrightarrow\N(\aM_{R}^\cf)\simeq\Mod{R'} 
\]
is a colimit-preserving functor whose restriction along the Yoneda
embedding
\[
\N\C\to\Fun(\N\C^{\op},\Gpd_\i)\simeq\Gpd_{\i/\Line{R'}}
\]
is
equivalent to the inclusion of the $\i$-subcategory 
$\N\C\simeq\Line{R'}\to\Mod{R'}$. It follows from Corollary \ref{t-co-axiom-geom-thom-functor} that $\N (T^L)$ is equivalent to the ``geometric'' Thom spectrum functor of \S\ref{sec:units-via-infty}.
\end{proof}

\begin{Remark}
The argument also implies the following apparently more general result.
Recall from \S\ref{sec:colim-pres-funct} that any map $k\colon B\GL{R}
\to B\GL{R}$ defines a functor from the $\i$-category of spaces over
$B\GL{R}$ to the $\i$-category of $R$-modules, defined by sending $f\colon
X \to B\GL{R}$ to the colimit of the composite 
\begin{equation}\label{eq:co-param-mor-eq}
X^{\op} \xra{f} B\GL{R} \xra{k} B\GL{R} \to \Rmod.
\end{equation}
On the other hand, according to
Proposition~\ref{prop:algebraic-colimit} below, we can describe the
derived smash product from section \ref{sec:morita} associated to $k$
as the colimit of the composite 
\[
X^{\op} \xra{f} B\GL{R} \xra{k} B\GL{R} \xra{\splus} \Mod{\splus \GL{R}}
\xra{(-) \sma_{\splus \GL{R}} R} \Rmod.
\]
Since both functors are given by the formula $M(k\circ f)$, the Thom
$R$-module of $f$ composed with $k$, we conclude that these two
procedures are equivalent for any $k$, not just the identity. 
\end{Remark}

\subsection{The ``neo-classical'' Thom spectrum functor}
\label{sec:neo-classical-thom}

In this section we compare the Lewis-May operadic Thom
spectrum functor to the Thom spectrum functors discussed in this
paper.  Since the May-Sigurdsson construction of the Thom spectrum in
terms of a parametrized universal spectrum over $B\GL{S}$
\cite{MR2271789}[23.7.4] is easily seen to be equivalent to the
space-level Lewis-May description, this will imply that all of the
known descriptions of the Thom spectrum functor agree up to homotopy.
Our comparison proceeds by relating the Lewis-May model to the
quotient description of Proposition~\ref{quotient}.

We begin by briefly reviewing the Lewis-May construction of the Thom
spectrum functor; the interested reader is referred to Lewis' thesis,
published as Chapter IX of \cite{LMS:esht}, and the excellent
discussion in Chapter 22 of \cite{MR2271789} for more details and
proofs of the foundational results below.  Nonetheless, we have tried
to make our discussion relatively self-contained.

The starting point for the Lewis-May construction is an explicit
construction of $\GL{S}$ in terms of a diagrammatic model of infinite
loop spaces.  Let $\sI_c$ be the symmetric monoidal category of finite
or countably infinite dimensional real inner product spaces and linear
isometries.  Define an $\sI_c$-space to be a continuous functor from
$\sI_c$ to spaces.  The usual left Kan extension construction (i.e., Day convolution) gives
the diagram category of $\sI_c$-spaces a symmetric monoidal
structure.  It turns out that monoids and commutative monoids for this
category model, respectively, $\ainfty$ and $\einfty$ spaces; for technical felicity,
we focus attention on the commutative monoids which satisfy two
additional properties:
\begin{enumerate}
\item The map $T(V) \to T(W)$ associated to a linear isometry $V \to
  W$ is a homeomorphism onto a closed subspace.
\item Each $T(W)$ is the colimit of the $T(V)$, where $V$ runs over
  the finite dimensional subspaces of $W$ and the maps in the colimit
  system are restricted to the inclusions.
\end{enumerate}
Denote such a functor as an $\sI_c$-FCP (functor with cartesian
product) \cite[23.6.1]{MR2271789}; the requirement that $T$ be a
diagrammatic commutative monoid implies the existence of a ``Whitney
sum'' natural transformation $T(U) \times T(V) \to T(U \oplus V)$.
This terminology is of course deliberately evocative of the notion of
$FSP$ (functor with smash product), which is essentially an orthogonal
ring spectrum \cite{MR1806878}.

An $\sI_c$-FCP gives rise to an $\einfty$ space structured by the
linear isometries operad $\sL$; specifically, $T(\R_\infty) = \colim_V T(V)$
is an $\sL$-space with the operad maps induced by the Whitney sum
\cite[1.9]{MQRT:ersers}, \cite[23.6.3]{MR2271789}.  In fact, as
alluded to above one can set up a Quillen equivalence between the
category of $\sI_c$-FCP's and the category of $\einfty$ spaces,
although we do not discuss this matter herein (see \cite{Lind} for a
nice treatment of this comparison).

Moving on, we now focus attention on the $\sI_c$-FCP specified by
taking $V \subset \R^\infty$ to the space of based homotopy
self-equivalences of $S^V$; this is classically denoted by $F(V)$.
Passing to the colimit over inclusions, $F(\R^\infty) = \colim_V F(V)$
becomes a $\sL$-space which models $\GL{S}$ --- this is essentially one
of the original descriptions from \cite{MQRT:ersers}.  Furthermore,
since each $F(V)$ is a monoid, applying the two-side bar construction
levelwise yields an FCP specified by $V \mapsto BF(V)$; here $BF(V)$
denotes the bar construction $B(*,F(V),*)$, and the Whitney sum
transformation is defined using the homeomorphism $B(*, F(V), *)
\times B(*, F(W), *) \cong B(*, F(V) \times F(W), *)$.  The colimit
$BF(\R^\infty)$ provides a model for $B\GL{S}$.

Now, since $F(V)$ acts on $S^V$, we can also form the two-sided bar
construction $B(*,F(V),S^V)$, abbreviated $EF(V)$, and there is a
universal quasifibration
\[
\pi_V \colon EF(V) = B(*,F(V),S^V) \longrightarrow B(*,F(V),*) = BF(V)
\]
which classifies
spherical fibrations with fiber $S^V$.  Given a map $X \to
BF(\R^\infty)$, by pulling back subspaces $BF(V) \subset
BF(\R^\infty)$ we get an induced filtration on $X$; denote the space
corresponding to pulling back along the inclusion of $V \in \R^\infty$
by $X(V)$ \cite[IX.3.1]{LMS:esht}.

Denote by $Z(V)$ the pullback
\[
\xymatrix{
X(V) \ar[r] & BF(V) & \ar[l] EF(V).
}
\]
The $V$th space of the Thom prespectrum is then obtained by taking the
Thom space of $Z(V) \to X(V)$, that is by collapsing out the section
induced from the base point inclusion $* \to S^V$; denote the resulting
prespectrum by $TF$ (see \cite[IX.3.2]{LMS:esht}, and note that some work is
involved in checking that these spaces in fact assemble into a
prespectrum).

Next, we will verify that the prespectrum $TF$ associated to the
identity map on $BF(\R^\infty)$ is stably equivalent to the homotopy
quotient $S/\GL{S} \heq S/F(\R^\infty)$.  For a point-set description of
this homotopy quotient, it follows from~\cite[3.9]{units-sma}
that the category of EKMM (commutative) $S$-algebras is tensored over
(commutative) monoids in $*$-modules: the tensor of a monoid in
$*$-modules $M$ and an $S$-algebra $A$ is $\Sigma^{\infty}_{\L+} M \sma A$, with
multiplication 
\begin{align*}
(\Sigma^\infty_{\L+} M \sma A) \sma (\Sigma^\infty_{\L+} M \sma A) \cong &(\Sigma^\infty_{\L+} M \sma \Sigma^\infty_{\L+}
M) \sma (A \sma A) \cong \\
&(\Sigma^\infty_{\L+} (M \boxtimes M)) \sma (A \sma A) \to
(\Sigma^\infty_{\L+} M) \sma A. 
\end{align*}

Thus, we can model the homotopy quotient as a bar construction in the
category of (commutative) $S$-algebras.  However, we can also describe
the homotopy quotient as $\colim_V S/F(V)$, where here we use the
structure of $F(V)$ as a monoid acting on $S^V$.  It is this
``space-level'' description we will employ in the comparison below.

We find it most convenient to reinterpret the Lewis-May construction
in this situation, as follows:  The Thom space in this case is by
definition the cofiber $(EF(V),BF(V))$ of the inclusion $BF(V) \to
EF(V)$ induced from the base point inclusion $* \to S^V$.
Now, 
$$
BF(V)\heq */F(V)$$
and similarly
$$
EF(V)\heq S^V/F(V).
$$
Hence the Thom space is likewise the cofiber $(S^V,*)/F(V)$ of the
inclusion $*\to S^V$, viewed as a {\em pointed} space.

More generally, we can regard the prespectrum $\{MF(V)\}$ as
equivalently described as
\[
MF(V)\eqdef S^V/F(V),
\]
the homotopy quotient of the {\em pointed} space $S^V$ by $F(V)$ via
the canonical action, with structure maps induced from the quotient
maps $S^V\to S^V/F(V)$ together with the pairings 
\begin{align*}
MF(V)\land MF(W)\heq \, &S^V/F(V)\land S^W/F(W)\longrightarrow\\
&\,\, S^{V\oplus W}/F(V)\times F(W)\longrightarrow S^{V\oplus W}/F(V \oplus
W), 
\end{align*}
where $F(V)\times F(W)\to F(V \oplus W)$ is the Whitney sum map of
$F$.  It is straightforward to check that the structure maps in terms
of the bar construction described in \cite[IX.3.2]{LMS:esht} realize
these structure maps.

The associated spectrum $MF$ is then the colimit $\colim_V
S/F(V) \heq S/F(\R^\infty)$.  A key point is that the Thom spectrum
functor can be described as the colimit over shifts of the Thom spaces
\cite[IX.3.7,IX.4.4]{LMS:esht}:
\[
MF = \colim_V \Sigma^{-V} \Sigma^\i MF(V).
\]
Furthermore, using the bar construction we can see that the spectrum
quotient $(\Sigma^V S) / F(V)$ is equivalent to $\Sigma^\i S^V /
F(V)$.  Putting these facts together, we have the following chain of
equivalences: 
\begin{eqnarray*}
MF & = & \colim_V\Sigma^{-V}\Sigma^\i
MF(V) = \colim_V\Sigma^{-V}\Sigma^\i S^V /F(V) \\ 
& \heq & \colim_V
\Sigma^{-V} (\Sigma^V S)/F(V) \heq \colim_V (\Sigma^{-V} \Sigma^V
S)/F(V) \htp S/F(\R^\infty).
\end{eqnarray*}

More generally, a slight elaboration of this argument implies the
following:

\begin{Proposition}\label{L-Mquotient}
The Lewis-May Thom spectrum $MG$ associated to a group-like $A_\i$ map
$\varphi:G\to GL_1 S$ modeled by the map of $\sI_c$-FCPs $G \to F$ is
equivalent to the spectrum $S/G$, the homotopy quotient of the sphere
by the action of $\varphi$.
\end{Proposition}

Note that any $A_\i$ map $X \to F(\R^{\infty})$ can be rectified to a map of
$\sI_c$-FCPs $X' \to F$ \cite{Lind}.

\begin{Corollary}\label{t-co-LM-thom-comp}
Given a map of spaces $f\colon X\to BGL_1 S$, write $M_\mathrm{LM}f$
for the spectrum associated to the Lewis-May Thom spectrum of $f$.
Then $M_\mathrm{LM}f\heq \Mgeo f$ as objects of the $\i$-category of
spectra.
\end{Corollary}

\begin{proof}
A basic property of the Thom spectrum functor $M_{\mathrm{LM}}$ is that it
preserves colimits \cite[IX.4.3]{LMS:esht}.  Thus, we can assume that
$X$ is connected.  In this case, $X\heq BG$ for some group-like
$A_\i$ space $G$, and $f\colon BG\to BGL_1 S$ is the delooping of an $A_\i$
map $G\to GL_1 S$. Hence $\Mgeo f\heq S/G$ by Proposition \ref{L-Mquotient}
and $M_\mathrm{LM}f\heq \Mgeo f$ by Proposition \ref{quotient}.
\end{proof}

\subsection{The algebraic Thom spectrum functor as a colimit}
\label{sec:algebr-thom-spectr}

We sketch another approach to the comparison of the ``geometric'' and
``algebraic'' Thom spectrum functors.  This approach has the
advantage of giving a direct comparison of the two functors.  It has
the disadvantage that it does 
not characterize the Thom spectrum functor among functors 
\[
   \TT_{/B\GL{R}} \to \Mod{R},
\]
and it does not exhibit the conceptual role played by Morita
theory.  Instead, it identifies both functors as colimits.

Suppose that $R$ is an $S$-algebra.   Let $\Rmod$ be the associated
$\i$-category of $R$-modules, let $\Rwe$ be the the sub-$\i$-groupoid
of $R$-lines, and let $j\colon \Rwe \to \Rmod$ denote the inclusion.
For a space $X$, the ``geometric'' Thom spectrum functor sends a
map $f\colon \Sing{X}^{\op} \to \Line{R}$ to 
\[
    \colim (\Sing{X}^{\op} \xra{\Sing{f}}\Line{R}  \xra{j}
    \Rmod).
\]

As in~\cite[\S 3]{units-sma}, let $G$ be a cofibrant
replacement of $\GL{R}$ as a monoid in $*$-modules.  By definition of
$\Line{R}$, we have an equivalence $\Sing{B_{\boxtimes}{G}}  \heq
\Rwe$.  But observe that we  also have a natural equivalence
\[
       \Sing{B_{\boxtimes}{G}} \heq \Line{G}.
\]
That is, let $\Mod{G} = \N({\EKMM_{G}}^{\cf})$  be the
$\i$-category of $G$-modules and let $\Line{G}$ be the maximal
$\i$-groupoid generated by the $G$-lines, i.e., $G$-modules which
admit a weak equivalence to $G$.  By construction, $\Line{G}$ is
connected, and so equivalent to $B\Aut (G)\heq B_{\boxtimes}{G}$.

Recall that we have an equivalence of $\i$-categories  
\begin{equation} \label{eq:75}
      \Gpd_{\i/\Line{G}} \heq \Mod{G}.
\end{equation}
The key observation is the following.  Let $k\colon \Line{G}\to \Mod{G}$
denote the tautological inclusion.  To a map of $\i$-groupoids
\[
   f\colon X^{\op} \to \Line{G},
\]
we can associate the $G$-module
\[
      P_{f} =  \colim (X^{\op} \xra{f} \Line{G}\xra{k} \Mod{G} ).
\]
Inspecting the proof of~\cite[2.2.1.2]{HTT} implies that the functor
$P\colon \Gpd_{\i/\Line{G}}\to\Mod{G}$ gives the equivalence
\eqref{eq:75}.

In other words, if $f\colon X\to B_{\boxtimes}{G}$ is a fibration of $*$-modules,
then we can form $P$ as in the pullback along $E_{\boxtimes}{G} \to
B_{\boxtimes}{G}$.  Alternatively, we can form 
\[
   \Sing{f}\colon \Sing{X}\to \Sing{B_{\boxtimes}{G}} \heq \Line{G},
\]
and then form $P_{\Sing{f}} = \colim (k \Sing{f}),$ and obtain
an equivalence of $G$-modules
\[
      P_{\Sing{f}} \heq P.
\]

\begin{Proposition}\label{prop:algebraic-colimit}
Let $f\colon X\to B_{\boxtimes}{G}$ be a fibration of $*$-modules.
The ``algebraic'' Thom spectrum functor sends $f$ to  
\[
   \colim (\Sing{X}^{\op} \xra{\Sing{f}} \Sing{B_{\boxtimes}{G}} \heq \Line{G}
   \xra{k} \Mod{G} \xra{\splus} \Mod{\splus G} \xra{\Smash_{\splus G}
   R} \Rmod ).
\]
\end{Proposition}

\begin{proof}
We have 
\begin{equation}\label{eq:1}
   P \heq \colim (\Sing{X}^{\op} \xra{\Sing{f}} \Sing{B_{\boxtimes}{G}} \heq \Line{G}
   \xra{k} \Mod{G}),
\end{equation}
and so 
\begin{align*}
    Mf = \splus P \Smash_{\splus G} R & \heq 
 \splus \colim (k\Sing{f}) \Smash_{\splus G}R \\
& \heq \colim (\splus k \Sing{f}) \Smash_{\splus G} R \\
& \heq \colim (\splus k \Sing{f} \Smash_{\splus G} R).
\end{align*}
\end{proof}

From this point of view, the coincidence of the two Thom spectrum
functors amounts to the fact that diagram 
\[
\xymatrix{
{X}^{\op} \ar[r]^{f}
 \ar[dr]_{f}
&
\Line{G}
 \ar[d]_{\heq}^{\splus (\slot) \Smash_{\splus G} R}
 \ar[rr]^{k}
& & 
{\Mod{G}}
 \ar[d]^{\splus (\slot) \Smash_{\splus G} R}
\\
& {\Rwe} 
 \ar[rr]_{j}
& & 
{\Rmod}
}
\]
evidently commutes.


\def\cprime{$'$}

\end{document}